\documentclass{article}
\usepackage{amsmath,amssymb,amsthm,color,bbm}

\title{The Modal Logic of Provability and Forcing}
\author{Taishi Kurahashi\footnote{Email: kurahashi@people.kobe-u.ac.jp}
and Rihito Takase\footnote{Email: 206X601X@stu.kobe-u.ac.jp}
\footnote{Graduate School of System Informatics, Kobe University, 1-1 Rokkodai, Nada, Kobe 657-8501, Japan.}}
\date{}

\theoremstyle{plain}
\newtheorem{thm}{Theorem}[section]
\newtheorem*{thm*}{Theorem}
\newtheorem{lem}[thm]{Lemma}
\newtheorem{prop}[thm]{Proposition}
\newtheorem{cor}[thm]{Corollary}

\newtheorem*{fact*}{Fact}

\newtheorem*{prob*}{Problem}
\newtheorem{cl}{Claim}
\newtheorem*{cl*}{Claim}
\newtheorem*{scl*}{Subclaim}

\theoremstyle{definition}
\newtheorem{defn}[thm]{Definition}

\newcommand{\ZFC}{\mathrm{ZFC}}
\newcommand{\PF}{\mathbf{PF}}
\newcommand{\PFo}{\mathbf{PF}^\omega}
\newcommand{\GL}{\mathbf{GL}}

\newcommand{\SF}{\mathbf{S4.2}}
\newcommand{\GT}{\GL \otimes \mathbf{Triv}}

\newcommand{\CU}{\mathrm{L}}

\newcommand{\accp}{\sqsubset}
\newcommand{\accf}{\preccurlyeq}

\newcommand{\forces}{\mathrel{\Vdash \hspace{-6.5pt} \Vdash}}
\newcommand{\Pforces}[1]{\mathrel{\forces_{#1}}}

\newcommand{\p}{\mathsf{p}}
\newcommand{\f}{\mathsf{f}}
\newcommand{\pf}{\mathsf{pf}}

\newcommand{\Sub}{\mathrm{Sub}}
\newcommand{\Subb}{\overline{\Sub}}
\newcommand{\Th}{\mathrm{Th}}
\newcommand{\mc}{\mathcal}
\newcommand{\ul}{\underline}

\newcommand{\BP}{\Box_{\mathsf{p}}}
\newcommand{\DP}{\Diamond_{\mathsf{p}}}
\newcommand{\BF}{\Box_{\mathsf{f}}}
\newcommand{\DF}{\Diamond_{\mathsf{f}}}

\newcommand{\Prov}{\textcolor[gray]{0.7}{\blacksquare}_{\mathsf p}}
\newcommand{\Con}{\textcolor[gray]{0.7}{\blacklozenge}_{\mathsf p}}
\newcommand{\ForceB}{\textcolor[gray]{0.7}{\blacksquare}_{\mathsf f}}
\newcommand{\ForceD}{\textcolor[gray]{0.7}{\blacklozenge}_{\mathsf f}}

\newcommand{\LS}{\mc L_{\in}}
\newcommand{\LP}{\mc L_{\p}}
\newcommand{\LF}{\mc L_{\f}}
\newcommand{\LPF}{\mc L_{\pf}}

\newcommand{\gp}{g_{\p}}
\newcommand{\gf}{g_{\f}}
\newcommand{\gpf}{g_{\pf}}

\begin{document}

\maketitle

\begin{abstract}
    Solovay's arithmetical completeness theorem states that the modal logic of provability coincides with the modal logic $\GL$.
    Hamkins and L\"owe studied the modal logical aspects of set theoretic multiverse and proved that the modal logic of forcing is exactly the modal logic $\SF$.
    We explore the interaction between the notions of provability and forcing in terms of modal logic. 
    We introduce the bimodal logic $\PF$ and prove that the modal logic of provability and forcing is exactly $\PF$. 
    We also introduce the bimodal logic $\PFo$ and prove that $\PFo$ is exactly the modal logic of provability and forcing true in $\omega$-models of set theory. 
 \end{abstract}

\section{Introduction}

Set theorists have been interested in and studied the structure of models (universes) of the set theory $\ZFC$. 
If $\ZFC$ is consistent, then there exist numerous universes with different properties, and their overall structure, namely the multiverse, has been investigated in recent years \cite{Ham12}. 
In the study of the multiverse, there has been interest in the nature of relational structures on the multiverse. 
A typical example of such a structure is one based on forcing extension. 

Hamkins and L\"owe \cite{HL08} studied this relational structure by analyzing the behavior on $\ZFC$ of the modal operator $\BF$, where $\BF \varphi$ means ``$\varphi$ is true in all forcing extensions''. 
Here, the subscript $\mathsf{f}$ stands for `forcing'. 
Hamkins and L\"owe called the set of all $\ZFC$-verifiable $\BF$-principles the \textit{modal logic of forcing}, and then proved that it is exactly the modal logic $\SF$ whenever $\ZFC$ is consistent. 

Another typical example of a method of obtaining models is one based on completeness theorem. 
The modal operator $\BP$ corresponding to this case is the $\ZFC$-provability, where the subscript $\mathsf{p}$ stands for `provability'. 
The set of all $\ZFC$-verifiable $\BP$-principles is known as the \textit{provability logic of $\ZFC$} (cf.~\cite{Boo93,Smo85}). 
It is then well-known as Solovay's arithmetical completeness theorem \cite{Sol76} that the provability logic of $\ZFC$ is exactly the modal logic $\GL$ whenever $\ZFC$ is $\Sigma_1$-sound in the sense of arithmetic. 

Hamkins and L\"owe refer to Solovay's work at the beginning of their paper, however, the modal logical studies of $\BP$ and $\BF$ are currently placed in different contexts. 
Nevertheless, forceability and provability are not irrelevant. 
For example, the following fact is stated in usual textbooks of set theory: For any sentence $\varphi$ of set theory, if $\varphi$ is provable in $\ZFC$, then ``$\varphi$ is true in all forcing extensions'' is also provable in $\ZFC$ (cf.~\cite{Jec03,Kun80}). 
By using the modal operators $\BP$ and $\BF$, this fact is expressed as $\BP \varphi \to \BP \BF \varphi$. 
Our main goal in this paper is to investigate this kind of interaction between these two operators, that is, the $\ZFC$-provable $(\BP, \BF)$-principles. We then call the set of all such principles the \textit{modal logic of provability and forcing}. 
Through this investigation, we aim to clarify the relationship between forceability and provability.

For this purpose, we introduce the bimodal logic $\PF$ in the language having two modal operators $\BP$ and $\BF$. 
The modal axioms of $\PF$ are those of $\GL$ for $\BP$, those of $\SF$ for $\BF$, and the following three new interaction axioms: 
\begin{itemize}
    \item $\BP A \to \BF \BP A$
    \item $\DP A \to \BF \DP A$
    \item $\BP A \to \BP \BF A$
\end{itemize}
Our main theorem is the following: 

\begin{thm*}[Main Theorem]
If $\ZFC$ is $\Sigma_1$-sound in the sense of arithmetic, then the modal logic of provability and forcing is exactly the modal logic $\PF$. 
\end{thm*}

Our proof strategy of the main theorem is to combine the methods of Solovay and Hamkins--L\"owe. 
For this, we first explore the Kripke semantics of $\PF$. 
Kripke frames for $\PF$ have two accessibility relations $\accp$ and $\accf$ respectively corresponding to the modal operators $\BP$ and $\BF$, and are two-layered according to each of these relations. 
Roughly speaking, we say that a Kripke frame $\mc F = (W, \accp, \accf)$ is a \textit{nice $\PF$-frame} iff $W$ is a disjoint union of Kripke frames of $\SF$ with respect to $\accf$, and $\accp$ is thought as a $\GL$-accessibility relation between these $\SF$-frames. 
Actually, $\accp$ is a relation on $W$, and the above definition means that for any $\SF$-frames $\mc C_0$, $\mc C_1 \subseteq W$ and any $u, u' \in \mc C_0$ and $v, v' \in \mc C_1$, we have that $u \accp v$ if and only if $u' \accp v'$. 
Our nice $\PF$-frames resemble to stratified frames of polymodal logic $\mathbf{J}$ of provability which are presented in \cite{Bek10}. 
We then establish the following theorem: 
\begin{thm*}[The finite frame property of $\PF$]
The logic $\PF$ is sound and complete with respect to the class of all finite nice $\PF$-frames. 
\end{thm*}

We further prove that $\PF$ is characterized by a class of finite nice $\PF$-frames that are in some sense well-formed (see Theorem \ref{ffp_PF} for details). 
Our proof of Main Theorem is done by embedding such well-formed finite nice $\PF$-frames into $\ZFC$ by applying Solovay's and Hamkins--L\"owe's proof methods to each layer of such frames, respectively. 
The finite frame property of $\PF$ seems interesting on its own, not just as an intermediate step to our purpose. Furthermore, we feel that the two-layered structure of nice $\PF$-frames also provides deep insight into the relationship between forceability and provability in the relational structure of the multiverse.

We say a model $M$ of $\ZFC$ an \textit{$\omega$-model} iff $\omega^M$ is standard. 
Solovay \cite{Sol76} also studied the \textit{truth provability logic} that is the set of all $\BP$-principles true in all $\omega$-models. 
Solovay introduced the logic $\mathbf{S}$ that is a non-normal extension of $\GL$, and proved that the truth provability logic is exactly $\mathbf{S}$ (see also \cite{Boo93}). 
Inspired by this work, we also introduce the logic $\PFo$ whose axioms are all theorems of $\PF$ and formulas of the form $\BP A \to \BF A$ and whose sole rule is modus ponens. 
We then prove the following theorem concerning $\PFo$: 

\begin{thm*}
If an $\omega$-model of $\ZFC$ exists, then the $(\BP, \BF)$-principles true in all $\omega$-models are exactly the theorems of $\PFo$. 
\end{thm*}

The present paper is organized as follows. 
In Section \ref{sec_pre}, we introduce some notions and facts needed to read the present paper.
In Section \ref{sec_pf}, we introduce the bimodal logic $\PF$ and prove that every theorem of the logic is a $\ZFC$-verifiable principle of provability and forcing. 
Section \ref{sec_ks} is devoted to the development of Kripke semantics of $\PF$, and in particular we prove the finite frame property of $\PF$.
Our main theorem stating that $\PF$ is exactly the modal logic of provability and forcing, is proved in Section \ref{sec_mt}.
In Section \ref{sec_pfo}, we introduce the logic $\PFo$ and prove that $\PFo$ is exactly the modal logic of provability and forcing true in $\omega$-models. 
In Appendix \ref{App1}, we prove that $\PF$ is a conservative extension of $\GL$ and that $\PFo$ is a conservative extension of both $\mathbf{S}$ and $\SF$. 
A version of conjunctive normal form theorem for $\PF$ is developed in Appendix \ref{App2}.

\section{Preliminaries}\label{sec_pre}

In this section, we introduce some notions and facts of set theory and modal logic which are needed to read the present paper. 

\subsection{Prerequisites for set theory}

The language $\{\in\}$ of set theory is denoted by $\mc L_\in$. 
In this subsection, $M$ denotes a countable transitive model of (some enough large fragment of )$\ZFC$.
Forcing method is very important in set theory to prove the consistency of several statements of set theory with $\ZFC$.

We call the triple $\langle \mathbb{P}, \leq, \mathbbm{1} \rangle \in M$ a \textit{forcing poset} if $\mathbb{P} \in M$, $\mathbbm{1} \in \mathbb{P}$, and $\leq$ is a partial order on $\mathbb{P}$ whose maximum element is $\mathbbm{1}$.
We abbreviate $\langle \mathbb{P}, \leq, \mathbbm{1} \rangle$ as $\mathbb{P}$ unless there is any risk of misunderstanding.
Informally speaking, with a forcing poset $\mathbb{P}$ and a generic filter $G \subseteq \mathbb{P}$ over $M$, we can construct a model $M[G]$ of $\ZFC$, and it is shown that $M[G]$ is minimum among models $M'$ of $\ZFC$ such that $M \subseteq M'$ and $G \in M'$ (see \cite{Jec03,Kun80} for details).
We call $M[G]$ a \textit{forcing extension} of $M$ by $G$.
Formally, we can investigate the behavior of $M[G]$ by using the \textit{forcing relation}.
For each $\mc L_{\in}$-sentence $\varphi$, we can define the $\mc L_{\in}$-formula $\forces_{\mathbb{P}} \varphi$ with a parameter $\mathbb{P}$, which means that every forcing extension defined by $\mathbb{P}$ satisfies $\varphi$. 
See \cite[Lemma IV.2.24 and Lemma IV.2.25]{Kun80} for the detailed definition.
We say \textit{$\mathbb{P}$ forces $\varphi$} if and only if $\Pforces{\mathbb{P}} \varphi$.





We say that an $\LS$-formula $\varphi$ is \textit{arithmetical} if it is of the form $\psi^\omega$ for some $\LS$-sentence $\psi$, where $\psi^\omega$ is the $\LS$-formula obtained from $\psi$ by replacing all occurrences of $\forall x$ and $\exists x$ by $\forall x \in \omega$ and $\exists x \in \omega$, respectively. 
Arithmetical sentences are $\Delta_0$-sentences with using $\omega$ as a parameter.
Since forcing does not change $\omega$,
we obtain the following proposition.

\begin{prop}[Cf.~{\cite[Lemma 14.21]{Jec03}}] \label{arithAbsolute_forcing}
    For any arithmetical $\mc L_{\in}$sentence $\varphi$,
    \[\ZFC \vdash \forall \mathbb{P} : \text{forcing poset}. \; (\varphi \leftrightarrow \Pforces{\mathbb{P}} \varphi).\]
\end{prop}


Inner models of set theory also play significant roles.
In particular, \textit{the universe of constructible sets} $\CU$ is one of important inner models.
See \cite[Definition II.6.1]{Kun80} for the detailed definition.
Let $\mathrm{V} = \CU$ be the $\mc L_{\in}$-sentence $\forall x\, (x \in \CU)$ which means that all sets are constructible.
For $\mc L_{\in}$-sentence $\varphi$, $\varphi^{\CU}$ denotes the sentence obtained from $\varphi$ by replacing all occurrences of $\forall x$ and $\exists x$ by $\forall x \in \CU$ and $\exists x \in \CU$, respectively. 
Every model $M$ of $\ZFC$ has $\CU$ as its inner model, which is denoted by $\CU^M$.
Then, it is shown that $\CU^M \models \ZFC + \mathrm{V} = \CU$. 

Since $\ZFC \vdash \forall x (x =\omega^{\CU} \leftrightarrow x = \omega)$,
the class $\CU$ has the following absoluteness property.



\begin{prop}[Cf.~{\cite[Lemma I.16.2]{Kun80}}] \label{arithAbsolute_forcing_L}
    For any arithmetical $\mc L_{\in}$-sentence $\varphi$,
    \[ \ZFC \vdash \varphi \leftrightarrow \varphi^{\CU}. \]
\end{prop}

\subsection{The modal logic of provability}

Let $\LP$ denote the language of modal propositional logic having countably many propositional variables, propositional connectives $\neg, \land, \lor, \to$, and one unary modal operator $\BP$. 
We assume that the unary modal operator $\DP$ is the abbreviation of $\neg \BP \neg$. 

\begin{defn}
The logic $\GL$ in the language $\LP$ is defined as follows: 
    \begin{itemize}
        \item The axioms of $\GL$ are:
        \begin{enumerate}
            \item All propositional tautologies in the language $\LP$.  
            \item $\BP (A \to B) \to (\BP A \to \BP B)$. 
            \item $\BP (\BP A \to A) \to \BP A$. 
        \end{enumerate}
        \item The inference rules of $\GL$ are modus ponens and necessitation $\dfrac{A}{\BP A}$. 
    \end{itemize}
\end{defn}

\begin{defn}
    We say that a tuple $(W, \accp)$ is a \textit{Kripke frame} iff $W$ is a non-empty set and $\accp$ is a binary relation on $W$. 
We say that a tuple $(W, \accp, \Vdash)$ is a \textit{Kripke model} iff $(W, \accp)$ is a Kripke frame and $\Vdash$ is a binary relation between $W$ and the set of all $\LP$-formulas fulfilling the usual conditions of satisfaction relation for propositional connectives and the following condition: 
\[
    x \Vdash \BP A \iff \forall y \in W\, (x \accp y \Rightarrow y \Vdash A). 
\]
We say an $\LP$-formula $A$ is \textit{valid} in a Kripke model $(W, \accp, \Vdash)$ iff $x \Vdash A$ for all $x \in W$. 
We also say that $A$ is valid on a Kripke frame $\mc F = (W, \accp)$, denoted by $\mc F \vDash A$, iff $A$ is valid in all Kripke models $(W, \accp, \Vdash)$ based on the frame $(W, \accp)$. 
\end{defn}

We say that a Kripke frame $\mc F$ is a \textit{$\GL$-frame} iff all theorems of $\GL$ are valid on $\mc F$. 
A Kripke frame $(W, \accp)$ is said to be \textit{conversely well-founded} iff there is no infinite $\accp$-chain of elements of $W$. 

\begin{prop}[Cf.~Boolos {\cite[Theorem 4.10]{Boo93}}]
For any Kripke frame $\mc F = (W, \accp)$, the following are equivalent:
\begin{enumerate}
    \item $\mc F$ is a $\GL$-frame. 
    \item $\accp$ is transitive and conversely well-founded. 
\end{enumerate}
\end{prop}

A $\GL$-frame $(W, \accp)$ is said to be \textit{rooted} iff there exists an element $r \in W$ such that $r \accp x$ for all $x \in W \setminus \{r\}$. 
We call such an element $r$ the \textit{root} of the frame. 
The logic $\GL$ enjoys the following finite frame property. 

\begin{thm}[The finite frame property of $\GL$ (cf.~Boolos {\cite[Chapter 5]{Boo93}})]\label{ffp_GL}
For any $\LP$-formula $A$, the following are equivalent:
\begin{enumerate}
    \item $\GL \vdash A$. 
    \item $A$ is valid on all $\GL$-frames. 
    \item $A$ is valid on all finite rooted $\GL$-frames. 
    \item $A$ is true in the root of all finite rooted $\GL$-models. 
\end{enumerate}
\end{thm}

The logic $\GL$ is known as the modal logic of provability. 
Syntactical notions can be formalized in set theory by using arithmetical $\LS$-formulas. 
An arithmetical $\LS$-formula $\mathrm{Pr}_{\ZFC}(x)$ is called a \textit{provability predicate of $\ZFC$} iff it is a natural formalization of the $\ZFC$-provability. 
Throughout the paper, the sentence $\mathrm{Pr}_{\ZFC}(\varphi)$ is abbreviated by $\Prov \varphi$. 
Then, it is known that the following theorem holds: 

\begin{thm}[Derivability conditions and formalized L\"ob's theorem \cite{Lob55} (see also \cite{Boo93}]\label{DC}
For any $\LS$-sentences $\varphi$ and $\psi$: 
\begin{enumerate}
    \item If $\ZFC \vdash \varphi$, then $\ZFC \vdash \Prov \varphi$, 
    \item $\ZFC \vdash \Prov (\varphi \to \psi) \to (\Prov \varphi \to \Prov \psi)$, 
    \item $\ZFC \vdash \Prov \varphi \to \Prov \Prov \varphi$, 
    \item $\ZFC \vdash \Prov (\Prov \varphi \to \varphi) \to \Prov \varphi$. 
\end{enumerate}
\end{thm}
The logic $\GL$ is formulated to capture the $\ZFC$-verifiable modal principles concerning the provability predicate $\Prov$ of $\ZFC$. 
In particular, the axiom $\BP( \BP A \to A) \to \BP A$ of $\GL$ corresponds to the last clause of Theorem \ref{DC} which is the formalized version of L\"ob's theorem. 

To connect $\GL$ and set theory more precisely, we introduce the notion of translation. 
A mapping from the set of all propositional variables to a set of $\LS$-sentences is called a \textit{translation}. 
Each translation $g$ is uniquely extended to the mapping $\gp$ from the set of all $\LP$-formulas to a set of $\LS$-sentences by the following clauses: 
\begin{enumerate}
    \item $\gp$ commutes with each propositional connective, 
    \item $\gp(\BP A)$ is $\Prov \gp(A)$. 
\end{enumerate}

We say an extension $T$ of $\ZFC$ is \textit{arithmetically $\Sigma_1$-sound} iff for any $\Sigma_1$ $\LS$-sentence $\varphi$, we have $\omega \vDash \varphi$ whenever $T \vdash \varphi^\omega$.
We are ready to state Solovay's theorem. 

\begin{thm}[Solovay \cite{Sol76}] \label{completeness_GL}
If $\ZFC$ is arithmetically $\Sigma_1$-sound, then for any $\LP$-formula $A$, the following are equivalent: 
\begin{enumerate}
    \item $\GL \vdash A$. 
    \item $\ZFC \vdash \gp(A)$ for all translations $g$. 
\end{enumerate}
\end{thm}

In his proof of Solovay's arithmetical completeness theorem, an ingenious method of embedding finite rooted $\GL$-frames into arithmetic was presented. 

\begin{thm}\label{SS}
For any finite $\GL$-frame $\mc F = (W, \accp)$ with the root $r$, there exist arithmetical $\LS$-sentences $\{\lambda_i\}_{i \in W}$ satisfying the following conditions: for $i, j \in W$,
\begin{enumerate}
    \item if $i \neq j$, then $\ZFC \vdash \lambda_i \to \neg \lambda_j$, 
    \item $\ZFC \vdash \bigvee_{k \in W} \lambda_k$,
    \item if $i \neq r$, then $\ZFC \vdash \lambda_i \to \Prov \bigvee_{i \accp k} \lambda_k$, 
    \item if $i \accp j$, then $\ZFC \vdash \lambda_i \to \Con \lambda_j$, 
    \item if $i \neq r$, then $\ZFC \vdash \lambda_i \to \Prov \neg \lambda_i$,
    \item if $\ZFC$ is arithmetically $\Sigma_1$-sound, then $\ZFC \nvdash \neg \lambda_i$. 
\end{enumerate}
\end{thm}
We call such sentences $\{\lambda_i\}_{i \in W}$ \textit{Solovay sentences of $\mc F$}.

\subsection{The modal logic of forcing}

Let $\LF$ denote the language of modal propositional logic with one unary modal operator $\BF$. 

\begin{defn}
The logic $\SF$ in the language $\LF$ is defined as follows: 
    \begin{itemize}
        \item The axioms of $\SF$ are:
        \begin{enumerate}
            \item All propositional tautologies in the language $\LF$.  
            \item $\BF (A \to B) \to (\BF A \to \BF B)$. 
            \item $\BF A \to A$. 
            \item $\BF A \to\BF \BF A$. 
            \item $\DF \BF A \to \BF \DF A$. 
        \end{enumerate}
        \item The inference rules of $\BF$ are modus ponens and necessitation $\dfrac{A}{\BF A}$. 
    \end{itemize}
\end{defn}

In Kripke semantics, the symbol of binary relation for interpreting $\BP$ is written as $\accp$, but we will use $\accf$ as the symbol of binary relation for interpreting $\BF$.
We say that a Kripke frame $(W, \accf)$ is \textit{upward directed} iff for any $x, y, z \in W$, if $x \accf y$ and $x \accf z$, then there exists a $u \in W$ such that $y \accf u$ and $z \accf u$. 
As in the case of $\GL$, we say that a Kripke frame $\mc F$ is an $\SF$-frame iff all theorems of $\SF$ are valid on $\mc F$. 

\begin{prop}[Cf.~Chagrov and Zakharyaschev {\cite[Proposition 3.34]{CZ97}}]
For any Kripke frame $\mc F = (W, \accf)$, the following are equivalent:
\begin{enumerate}
    \item $\mc F$ is an $\SF$-frame. 
    \item $\accf$ is transitive, reflexive and upward directed. 
\end{enumerate}
\end{prop}

We say that an $\SF$-frame $(W, \accf)$ is \textit{rooted} iff there exists an element $r \in W$ such that $r \accf x$ for all $x \in W$. 
We call such an element $r$ a \textit{root element} of the frame.
Note that there may not be only one root element in $\SF$-frames.
The logic $\SF$ enjoys the following finite frame property.

\begin{thm}[The finite frame property of $\SF$ (cf.~Chagrov and Zakharyaschev {\cite[Theorem 5.33]{CZ97}})]\label{ffp_SF}
For any $\LF$-formula $A$, the following are equivalent:
\begin{enumerate}
    \item $\SF \vdash A$. 
    \item $A$ is valid on all $\SF$-frames. 
    \item $A$ is valid on all finite rooted $\SF$-frames. 
    \item $A$ is true in all root elements of all finite rooted $\SF$-models. 
\end{enumerate}
\end{thm}

Moreover, Hamkins and L\"owe proved that $\SF$ is characterized by a class of more well-formed finite rooted $\SF$-frames. 
For any $\SF$-frame $(W, \accf)$ and any $x, y \in W$, we write $x \approx y$ iff $x \accf y$ and $y \accf x$. 
It is easy to see that $\approx$ is an equivalence relation on $W$. 
Then, $\accf$ can be treated as a well-defined binary relation on $W / \approx$.  
We say that an $\SF$-frame is a \textit{pre-Boolean algebra} iff $(W / \approx, \accf)$ forms a Boolean algebra. 



\begin{thm}[Hamkins and L\"owe {\cite[Theorem 11]{HL08}}]\label{PBA}
For any $\LF$-formula $A$, the following are equivalent:
\begin{enumerate}
    \item $\SF \vdash A$. 
    \item $A$ is valid on all finite pre-Boolean algebras. 
    \item $A$ is true in all root elements of all rooted $\SF$-models based on finite pre-Boolean algebra. 
\end{enumerate}
\end{thm}

We abbreviate the $\LS$-sentence stating ``for any forcing poset $\mathbb{P}$, $\Pforces{\mathbb{P}} \varphi$'' as $\ForceB \varphi$.
Then, $\ForceD \varphi$ means ``there exists some forcing poset $\mathbb{P}$ which does not force $\varphi$''.
It is easily shown that this statement is $\ZFC$-provably equivalent to ``there exists some forcing poset $\mathbb{P}$ which forces $\neg \varphi$''.
The following theorem states that the modal logic of forcing is included in $\SF$.

\begin{thm}[Cf.~{\cite[Theorem 3]{HL08}}]\label{soundness_SF}
Let $\varphi$ and $\psi$ be any $\LS$-sentences. 
\begin{enumerate}
    \item If $\ZFC \vdash \varphi$, then $\ZFC \vdash \ForceB \varphi$. 
    \item $\ZFC$ proves the $\LS$-sentences $\ForceB (\varphi \to \psi) \to (\ForceB \varphi \to \ForceB \psi)$, $\ForceB \varphi \to \varphi$, $\ForceB \varphi \to \ForceB \ForceB \varphi$, and $\ForceD \ForceB \varphi \to \ForceB \ForceD \varphi$. 
\end{enumerate}
\end{thm}

As in the case of the modal logic of provability, we can uniquely extend a translation $g$, which is a mapping from the set of all propositional variables to a set of $\LS$-sentences, to the mapping $\gf$ from the set of all $\LF$-formulas to a set of $\LS$-sentences by the following clauses: 
\begin{enumerate}
    \item $\gf$ commutes with each propositional connective, 
    \item $\gf(\BF A)$ is $\ForceB \gf(A)$. 
\end{enumerate}
Then, Hamkins and L\"owe proved that the modal logic of forcing is exactly the modal logic $\SF$. 

\begin{thm}[Hamkins and L\"owe {\cite{HL08}}] \label{completeness_SF}
Suppose that $\ZFC$ is consistent. 
For any $\LF$-formula $A$, the following are equivalent: 
\begin{enumerate}
    \item $\SF \vdash A$. 
    \item $\ZFC \vdash \gf(A)$ for all translations $g$. 
\end{enumerate}
\end{thm}

As in the proof of Solovay's theorem, Theorem \ref{completeness_SF} is proved by embedding finite pre-Boolean algebras into $\ZFC$ by using the following theorem which is an analogue of Theorem \ref{SS}. 

\begin{thm}[Hamkins and L\"owe]\label{HLS}
For any finite pre-Boolean algebra $\mc F = (W, \accf)$ with a root element $r$, there exist $\LS$-sentences $\{\chi_a\}_{a \in W}$ satisfying the following conditions: for $a, b \in W$,
\begin{enumerate}
    \item if $a \neq b$, then $\ZFC \vdash \chi_a \to \neg \chi_b$, 
    \item $\ZFC \vdash \bigvee_{c \in W} \chi_c$,
    \item $\ZFC \vdash \chi_a \to \ForceB \bigvee_{a \accf c} \chi_c$, 
    \item if $a \accf b$, then $\ZFC \vdash \chi_a \to \ForceD \chi_b$, 
    \item $\ZFC + \mathrm{V = L} \vdash \chi_r$. 
    \item $\ZFC + \mathrm{V = L} \vdash \ForceD \chi_a$. 
\end{enumerate}
We call such sentences $\{\chi_a\}_{a \in W}$ \textit{HL sentences of $\mc F$}. 
\end{thm}

\section{The bimodal logic $\PF$ of provability and forcing}\label{sec_pf}

In this section, we introduce our bimodal logic $\PF$ and prove that all theorems of $\PF$ are verifiable in $\ZFC$. 
Let $\LPF$ denote the language of bimodal propositional logic with two unary modal operators $\BP$ and $\BF$.

\begin{defn}
The logic $\PF$ in the language $\LPF$ is defined as follows: 
    \begin{itemize}
        \item The axioms of $\PF$ are:
        \begin{enumerate}
            \item All propositional tautologies in the language $\LPF$. 
            \item All axioms of $\GL$. 
            \item All axioms of $\SF$.  
            \item $\BP A \to \BF \BP A$. 
            \item $\DP A \to \BF \DP A$. 
            \item $\BP A \to \BP \BF A$. 
            \end{enumerate}
        \item The inference rules of $\PF$ are modus ponens and necessitation rules $\dfrac{A}{\BP A}$ and $\dfrac{A}{\BF A}$. 
    \end{itemize}
\end{defn}

The $\ZFC$-verifiability of the $\GL$ and $\SF$-parts of $\PF$ follows from Theorems \ref{DC} and \ref{soundness_SF}. 
The last three interaction axioms for $\BP$ and $\BF$ of the logic $\PF$ are justified by the following two propositions.

\begin{prop}\label{prop_arith}
For any arithmetical $\LS$-sentence $\varphi$, we have $\ZFC \vdash \varphi \to \ForceB \varphi$. 
\end{prop}
\begin{proof}
This is immediate from Proposition \ref{arithAbsolute_forcing}. 
\end{proof}

\begin{prop}\label{prop_forcing}
For any $\LS$-sentence $\varphi$, we have $\ZFC \vdash \Prov \varphi \to \Prov \ForceB \varphi$. 
\end{prop}
\begin{proof}
    This proposition is obtained by proceeding a proof of the first statement of Proposition \ref{soundness_SF} in $\ZFC$. 
\end{proof}

As in the previous section, every translation $g$ can be uniquely extended to the mapping $\gpf$ from the set of all $\LPF$-formulas to a set of $\LS$-sentences by the following clauses: 
\begin{enumerate}
    \item $\gf$ commutes with each propositional connective, 
    \item $\gpf(\BP A)$ is $\Prov \gpf(A)$,  
    \item $\gpf(\BF A)$ is $\ForceB \gpf(A)$. 
\end{enumerate}

Since $\Prov \varphi$ and $\Con \varphi$ are arithmetical, by Propositions \ref{prop_arith} and \ref{prop_forcing}, we obtain the soundness of $\PF$ with respect to $\ZFC$-verifiability. 

\begin{prop}\label{soundness_PF}
For any $\LPF$-formula $A$ and any translation $g$, if $\PF \vdash A$, then $\ZFC \vdash \gpf(A)$.
\end{prop}


We prove the following proposition concerning HL sentences of finite pre-Boolean algebras, which will be used in our proofs in Sections \ref{sec_mt} and \ref{sec_pfo}.

\begin{prop}\label{con_HL}
Let $\mc F = (W, \accf)$ be any finite pre-Boolean algebra and $\{\chi_a\}_{a \in W}$ be HL sentences of $\mc F$. 
Then, for any arithmetical $\LS$-sentence $\varphi$ and $a \in W$, we have $\ZFC \vdash \Con \varphi \to \Con (\varphi \land \chi_a)$. 
\end{prop}
\begin{proof}
We argue in $\ZFC$. 
Suppose that $\Con \varphi$ holds. 
By the completeness theorem, we obtain a model $M$ of $\ZFC$ such that $M \vDash \varphi$. 
Since $\varphi$ is arithmetical, we have $\CU^M \vDash \varphi$ by Proposition \ref{arithAbsolute_forcing_L}. 
Since $\CU^M \vDash \ZFC + \mathrm{V} = \CU$ and $\ZFC + \mathrm{V} = \CU \vdash \ForceD \chi_a$ by Theorem \ref{HLS}.(6), we obtain $\CU^M \vDash \varphi \land \ForceD \chi_a$. 
By Proposition \ref{prop_arith}, we get $\CU^M \vDash \ForceB \varphi \land \ForceD \chi_a$, and hence $\CU^M \vDash \ForceD (\varphi \land \chi_a)$. 
We obtain $\Con \ForceD (\varphi \land \chi_a)$. 
By Proposition \ref{prop_forcing}, we conclude $\Con (\varphi \land \chi_a)$. 
\end{proof}

\section{Kripke semantics}\label{sec_ks}

In this section, we develop Kripke semantics of $\PF$. 
In particular, we prove the finite frame property of $\PF$ with respect to some class of well-formed finite $\PF$-frames. 

Since the language $\LPF$ has the two modal operators $\BP$ and $\BF$, Kripke frames of $\PF$ have two binary relations $\accp$ and $\accf$ corresponding to these operators respectively. 
As above, a Kripke frame $\mc F = (W, \accp, \accf)$ is called a \textit{$\PF$-frame} iff all theorems of $\PF$ are valid on $\mc F$. 
As in the cases of $\GL$ and $\SF$, the validity of $\PF$ on frames can be characterized by conditions of $\accp$ and $\accf$. 

\begin{thm}\label{fc}
Let $\mc F = (W, \accp, \accf)$ be any Kripke frame and $q$ be any propositional variable. 
\begin{enumerate}
    \item $\mc F \vDash \BP q \to \BF \BP q$ iff $\forall x, y, z \in W\, (x \accf y \accp z \Rightarrow x \accp z)$. 
    \item $\mc F \vDash \DP q \to \BF \DP q$ iff $\forall x, y, z \in W\, (x \accf y\ \&\ x \accp z \Rightarrow y \accp z)$. 
    \item $\mc F \vDash \BP q \to \BP \BF q$ iff $\forall x, y, z \in W\, (x \accp y \accf z \Rightarrow x \accp z)$. 
\end{enumerate}
\end{thm}
\begin{proof}
These clauses are easily shown. 
Proofs are left to the reader. 
\end{proof}

\begin{cor}
For any Kripke frame $\mc F = (W, \accp, \accf)$, $\mc F$ is a $\PF$-frame if and only if $\accp$ is transitive and conversely well-founded, $\accf$ is reflexive, transitive, and upward directed, and the right hand side of each clause of Theorem \ref{fc} holds. 
\end{cor}

Here we clarify the structure of $\PF$-frames more clearly. 
Actually, we show that every $\PF$-frame has a two-layered structure according to $\accp$ and $\accf$ where the inner layer consists of $\SF$-frames based on $\accf$, and the outer layer is a $\GL$-frame based on $\accp$.
Given any $\PF$-frame $\mc F = (W, \accp, \accf)$, let $\sim$ be the symmetric closure of $\accf$, namely $\sim$ is defined by: 
\[
    x \sim y : \iff x \accf y\ \text{or}\ y \accf x. 
\]
Also, let $\sim^+$ be the transitive closure of $\sim$. 
Then, $\sim^+$ is an equivalence relation on $W$, and for each $x \in W$, let $\mc C_x$ be the equivalence class of $x$. 
We call every equivalence class with respect to $\sim^+$ a \textit{cluster}. 
For every cluster $\mc C$, the Kripke frame $(\mc C, \accf)$ forms an $\SF$-frame, and so $W$ is thought of as a disjoint union of $\SF$-frames. 

Next, we explore the relation $\accp$. 

\begin{prop}\label{Prop_accp}
Let $\mc F = (W, \accp, \accf)$ be any $\PF$-frame and $x, y, z \in W$. 
\begin{enumerate}
    \item If $x \sim^+ y$ and $x \accp z$, then $y \accp z$. 
    \item If $x \sim^+ y$, then $x \not \accp y$. 
\end{enumerate}
\end{prop}
\begin{proof}
1. Suppose $x \sim^+ y$ and $x \accp z$. 
We find $u_0, u_1, \ldots, u_{n+1}$ such that $x = u_0$, $u_{n+1} = y$, and $u_i \sim u_{i+1}$ for $i \leq n$. 
We then prove by induction on $i \leq n$ that $u_i \accp z$. 
The case $i = 0$ is immediate from our supposition. 
Assume that $u_i \accp z$ holds for an $i < n+1$ and we show $u_{i+1} \accp z$. 
Since $u_i \sim u_{i+1}$, we have $u_i \accf u_{i+1}$ or $u_{i+1} \accf u_i$.
If $u_i \accf u_{i+1}$, then by Theorem \ref{fc}.(3), we have $u_{i+1} \accp z$. 
If $u_{i+1} \accf u_{i}$, then by Theorem \ref{fc}.(2), we get $u_{i+1} \accp z$. 
So, in either case, we obtain $u_{i+1} \accp z$. 
Consequently, we have $y = u_{n+1} \accp z$. 

2. Suppose $x \sim^+ y$. 
Since $y \not \accp y$, we have $x \not \accp y$ by (1). 
\end{proof}

We can treat the relation $\accp$ as the relation between clusters and elements of $W$ by defining $\mc C_x \accp y$ iff $x \accp y$. 
Proposition \ref{Prop_accp}.(1) says that this treatment is well-defined. 
Also, Proposition \ref{Prop_accp}.(2) says that if $\mc C_x \accp y$, then $\mc C_x \neq \mc C_y$. 
However, the relation $\accp$ cannot be thought of as a relation between clusters in general because even if $x \sim^+ y$ and $z \accp x$, it is not necessarily $z \accp y$ in general.    
To improve this situation, we introduce the following definition: 

\begin{defn}
We say that a $\PF$-frame $\mc F = (W, \accp, \accf)$ is \textit{nice} iff the following condition holds:
        \[
            \forall x, y, z \in W\, (x \accp z\ \&\ y \accf z \Rightarrow x \accp y). 
        \]
\end{defn}

\begin{prop}\label{Prop_nice}
Let $\mc F = (W, \accp, \accf)$ be any nice $\PF$-frame and $x, y, z, w \in W$. 
If $x \sim^+ y$, $z \sim^+ w$, and $x \accp z$, then $y \accp w$. 
\end{prop}
\begin{proof}
    Suppose $x \sim^+ y$, $z \sim^+ w$, and $x \accp z$. 
    By Proposition \ref{Prop_accp}, we have $y \accp z$. 
    Since $z \sim^+ w$, we find $u_0, u_1, \ldots, u_{n+1}$ such that $z = u_0$, $u_{n+1} = w$, and $u_i \sim u_{i+1}$ for $i \leq n$. 
    We prove by induction on $i \leq n$ that $y \accp u_i$. 
The case $i = 0$ is obvious. 
Assume that $y \accp u_i$ holds for some $i < n+1$ and we show $y \accp u_{i+1}$. 
If $u_i \accf u_{i+1}$, then by Theorem \ref{fc}.(1), we have $y \accp u_{i+1}$. 
If $u_{i+1} \accf u_{i}$, then by the niceness of $\mc F$, we obtain $y \accp u_{i+1}$. 
In either case, we conclude $y \accp u_{i+1}$. 
Therefore $y \accp u_{n+1} = w$. 
\end{proof}

Proposition \ref{Prop_nice} says that in a nice $\PF$ frame $\mc F = (W, \accp, \accf)$, the relation $\accp$ can be thought of as a relation between clusters, that is, the relation $\mc C_x \accp \mc C_y$ defined by $x \accp y$ is well-defined. 
To sum up, for every nice $\PF$-frame $\mc F = (W, \accp, \accf)$ and each cluster $\mc C$ of $\mc F$, we have that $(W / {\sim^+}, \accp)$ is a $\GL$-frame and $(\mc C, \accf)$ is an $\SF$-frame. 

We further introduce the following notions. 

\begin{defn}
We say that a nice $\PF$-frame $\mc F = (W, \accp, \accf)$ is \textit{rooted} iff the following two conditions hold:
\begin{itemize}
    \item There exists the $\accp$-root cluster of $\mc F$, that is, there exists an element $r \in W$ such that for any $x \in W$, either $\mc C_r = \mc C_x$ or $\mc C_r \accp \mc C_x$. 
    \item Every cluster $\mc C$ of $\mc F$ has a $\accf$-root element, that is, there exists an $x \in \mc C$ such that for any $y \in \mc C$, we have $x \accf y$. 
\end{itemize}
A nice rooted $\PF$-frame $\mc F$ is said to be \textit{PBA} iff $(\mc C, \accf)$ is a pre-Boolean algebra for any cluster $\mc C$ of $\mc F$. 
\end{defn}

We are ready to prove the finite frame property of $\PF$ with respect to the class of all finite rooted nice PBA $\PF$-frames.

\begin{thm}[The finite frame property of $\PF$]\label{ffp_PF}
For any $\LPF$-formula $A$, the following are equivalent: 
\begin{enumerate}
    \item $\PF \vdash A$. 
    \item $A$ is valid on all $\PF$-frames. 
    \item $A$ is valid on all finite rooted nice PBA $\PF$-frames.
    \item $A$ is true in all $\accf$-root elements of the $\accp$-root cluster of $\mc M$ for all finite rooted nice PBA $\PF$-models $\mc M$.
\end{enumerate}
\end{thm}
\begin{proof}
The implications $(1 \Rightarrow 2)$, $(2 \Rightarrow 3)$, and $(3 \Rightarrow 4)$ are easy. 
We prove the contrapositive of the implication $(4 \Rightarrow 1)$. 
Suppose $\PF \nvdash A$. 
We would like to find a finite rooted nice PBA $\PF$-model $\mc M^*$ and $\accf$-root element $r$ of the $\accp$-root cluster $\mc C_r$ of $\mc M^*$ such that $r$ falsifies $A$. 
For this purpose, we first provide a finite nice $\PF$-countermodel $\mc M$ of $A$. 
Then, we transform $\mc M$ to provide the desired model $\mc M^*$.

We say that a set $\Gamma$ of $\LPF$-formulas is \textit{$\PF$-consistent} iff for any finite subset $\Gamma_0$ of $\Gamma$, we have $\PF \nvdash \bigwedge \Gamma_0 \to \bot$, where $\bigwedge \Gamma_0$ is a conjunction of all elements of $\Gamma_0$. 
A $\PF$-consistent set $\Gamma$ is called \textit{$\PF$-maximally consistent} iff it is maximal among $\PF$-consistent sets. 
\begin{itemize}
    \item Let $W_0$ be the set of all $\PF$-maximally consistent sets. 
\end{itemize}
Of course $W_0$ is not a finite set, but we provide a finite set by dividing $W_0$ by some appropriate equivalence relation. 
Let $\Sub(A)$ be the set of all subformulas of $A$. 
We define the set $\Subb(A)$ of formulas as follows:
\[
    \Subb(A) : = \Sub(A) \cup \{- B \mid B \in \Sub(A)\} \cup \{\BF B, \neg \BF B \mid \BP B \in \Sub(A)\}, 
\]
where
\[
    - B = \begin{cases} C & \text{if}\ B\ \text{is of the form}\ \neg C, \\ \neg B & \text{otherwise.} \end{cases}
\]
Note that $\Sub(A)$ and hence $\Subb(A)$ are finite sets. 
We say that a $\PF$-consistent set $\Gamma$ is \textit{$A$-maximally $\PF$-consistent} iff $\Gamma \subseteq \Subb(A)$ and it is maximal among $\PF$-consistent subsets of $\Subb(A)$. 
\begin{itemize}
    \item Let $W_1$ denote the set of all $A$-maximally $\PF$-consistent sets. 
\end{itemize}
We have that $W_1$ is a finite set. 
We define the equivalence relation $\equiv$ on $W_0$ as follows: 
for $x, x' \in W_0$, let $x \equiv x'$ iff the following two conditions hold: 
\begin{enumerate}
    \item for any $y \in W_1$, we have $\DF \bigwedge y \in x$ if and only if $\DF \bigwedge y \in x'$, 
    \item for any $B \in \Subb(A)$, we have $B \in x$ if and only if $B \in x'$. 
\end{enumerate}
Since $W_1$ and $\Subb(A)$ are finite, we obtain that $W_0 / \equiv$ is finite. 
For $x \in W_0$, let $[x]$ denote the $\equiv$-equivalence class of $x$. 

We define the Kripke model $\mc M = (W, \accp, \accf, \Vdash)$ as follows: 
\begin{itemize}
    \item $W = \{([x], y) \mid [x] \in W_0 / \equiv, y \in W_1$, and $\DF \bigwedge y \in x\}$, 

    \item $([x], y) \accp ([x'], y')$ iff the following two conditions hold:
    \begin{enumerate}
    \item for any $\BP B \in \Subb(A)$, if $\BP B \in x$, then $\BP B, \BF B \in x'$, 
    \item for some $\BP C \in \Subb(A)$, we have $\BP C \notin x$ and $\BP C \in x'$.  
    \end{enumerate}

    \item $([x], y) \accf ([x'], y')$ iff the following two conditions hold:
    \begin{enumerate}
    \item $[x] = [x']$, 
    \item for any $\BF B \in \Subb(A)$, if $\BF B \in x$, then $\BF B \in x'$.  
    \end{enumerate}

    \item for any propositional variable $p$, $([x], y) \Vdash p$ iff $p \in y$. 

\end{itemize}
Notice that if $\BP B \in \Subb(A)$, then $\BF B \in \Subb(A)$. 
So by the definition of $\equiv$, we have that the set $W$ and the relation $\accf$ are well-defined. 
Since $W_0 / \equiv$ and $W_1$ are finite, $W$ is also a finite set. 
Notice that the establishment of the relation $([x], y) \accp ([x'], y')$ does not depend on $y$ and $y'$.

We prove the following useful lemma: 

\begin{lem}\label{useful}
For any $([x], y) \in W$ and $B \in \Subb(A)$, if $\BF B \in x$, then $B \in y$. 
\end{lem}
\begin{proof}
Suppose $([x], y) \in W$, $B \in \Subb(A)$, and $\BF B \in x$. 
Since $([x], y) \in W$, we have $\DF \bigwedge y \in x$, and hence $\DF (\bigwedge y \land B) \in x$. 
Since $\DF \bot \notin x$, we get $-B \notin y$. 
By the $A$-maximality of $y$, we obtain $B \in y$. 
\end{proof}

\begin{cl}\label{nice_PF}
The frame $\mc F = (W, \accp, \accf)$ of $\mc M$ is a finite nice $\PF$-frame. 
\end{cl}
\begin{proof}
It is easily shown that $\accp$ is transitive and conversely well-founded and $\accf$ is transitive and reflexive. 

For each $x \in W_0$, let $\mc C_{[x]} : = \{([x], y) \mid ([x], y) \in W\}$. 
We show that $\mc C_{[x]}$ forms a rooted cluster. 
Let $\ul x: = \{B \in \Subb(A) \mid B \in x\}$. 
Then, $\ul x \in W_1$. 
Since $\bigwedge \ul x \in x$, we have $\DF \bigwedge \ul x \in x$ because $\PF$ proves $\bigwedge \ul x \to \DF \bigwedge \ul x$. 
Hence $([x], \ul x) \in W$. 
It suffices to show that $([x], \ul x)$ is a $\accf$-root element of $\mc C_{[x]}$. 
Fix $([x], y) \in \mc C_{[x]}$. 
For a $\BF B \in \Subb(A)$, suppose $\BF B \in \ul x$. 
Then, $\BF B \in x$, and so $\BF \BF B \in x$. 
By Lemma \ref{useful}, we obtain $\BF B \in y$. 
By the definition of $\accf$, we conclude $([x], \ul x) \accf ([x], y)$. 

By the definition, the relation $\accp$ can be treated as a relation between clusters. 
So, the following three implications are easily verified: 
\begin{itemize}
    \item If $([x], y) \accf ([x], z) \accp ([x'], u) \accf ([x'], v)$, then $([x], y) \accp ([x'], v)$, 
    \item If $([x], y) \accf ([x], z)$ and $([x], y) \accp ([x'], u)$, then $([x], z) \accp ([x'], u)$, 
    \item If $([x], y) \accp ([x'], v)$ and $([x'], u) \accf ([x'], v)$, then $([x], y) \accp ([x'], u)$. 
\end{itemize}
Therefore, to show that $\mc F$ is a finite nice $\PF$-frame, it is sufficient to show that $\accf$ is upward directed. 

Assume $([x], y) \accf ([x], y_0)$ and $([x], y) \accf ([x], y_1)$. 
Then, $\DF \bigwedge y_0$ and $\DF \bigwedge y_1$ are in $x$. 
Suppose, towards a contradiction, that $\BF \neg \bigwedge \Gamma \in x$ for 
\[
    \Gamma : = \{\BF B \mid \BF B \in y_0\} \cup \{\BF C \mid \BF C \in y_1\}. 
\]
Then, 
\[
    \BF \bigl(\bigwedge_{\BF B \in y_0} \BF B \to \neg \bigwedge_{\BF C \in y_1} \BF C \bigr) \in x. 
\]
Since
\[
    \BF \BF \bigl(\bigwedge_{\BF B \in y_0} \BF B \to \neg \bigwedge_{\BF C \in y_1} \BF C \bigr) \in x, 
\]
we have
\[
    \BF \bigl(\bigwedge_{\BF B \in y_0} \BF \BF B \to \BF \neg \bigwedge_{\BF C \in y_1} \BF C \bigr) \in x. 
\]
So, we obtain
\[
    \BF \bigl(\bigwedge y_0 \to \BF \neg \bigwedge_{\BF C \in y_1} \BF C \bigr) \in x. 
\]
By combining this with $\DF \bigwedge y_0 \in x$, we get $\DF \BF \neg \bigwedge_{\BF C \in y_1} \BF C \in x$. 
By the axiom $\DF \BF D \to \BF \DF D$, we have $\BF \DF \neg \bigwedge_{\BF C \in y_1} \BF C \in x$, and equivalently $\BF \neg \bigwedge_{\BF C \in y_1} \BF \BF C \in x$. 
So, $\BF \neg \bigwedge_{\BF C \in y_1} \BF C \in x$, and thus $\BF \neg \bigwedge y_1 \in x$. 
This contradicts the $\PF$-consistency of $x$ because $\DF \bigwedge y_1 \in x$. 

We have shown that $\DF \bigwedge \Gamma \in x$. 
It is easily shown that $\PF$ proves $\DF B \to \DF (B \land C) \lor \DF(B \land \neg C)$ for any $\LPF$-formulas $B$ and $C$. 
So, we find some $z \in W_1$ such that $\Gamma \subseteq z$ and $\DF \bigwedge z \in x$ because $\Gamma \subseteq \Subb(A)$. 
Then, $([x], z) \in W$. 
By the definition of $\Gamma$, we obtain $([x], y_0) \accf ([x], z)$ and $([x], y_1) \accf ([x], z)$. 
Therefore, $\accf$ is upward directed. 
\end{proof}

By the proof of Claim \ref{nice_PF}, the mapping $[x] \mapsto \mc C_{[x]}$ is a bijection between $W_0 / \equiv$ and the set of all clusters of $\mc F$. 
Since $\accp$ can be thought as a relation between clusters, $\accp$ can be also treated as a relation on $W_0 / \equiv$. 
So, we write $[x] \accp [x']$ to mean $([x], y) \accp ([x'], y')$ for some $y, y' \in W_1$, and this is equivalent to $([x], \ul x) \accp ([x'], \ul x')$.

\begin{cl}\label{claim_tl}
    For any $B \in \Subb(A)$ and $([x], y) \in W$, we have $([x], y) \Vdash B$ if and only if $B \in y$. 
\end{cl}
\begin{proof}
We prove the claim by induction on the construction of $B$. 
If $B$ is a propositional variable, then the claim follows from the definition of $\Vdash$. 
The cases for $\land$, $\lor$, $\neg$, and $\to$ are easy. 
We only give proofs of the cases that $B$ is of the form $\BP C$ or $\BF C$, where the claim holds for $C$. 

\paragraph*{Case 1:} $B$ is $\BP C$. 

$(\Rightarrow)$: We prove the contrapositive. 
Suppose $\BP C \notin y$. 
By Lemma \ref{useful}, $\BF \BP C \notin x$, and hence $\BP C \notin x$ because $\PF \vdash \BP C \to \BF \BP C$. 
Suppose, towards a contradiction, that the set $\Gamma$ defined as follows is not $\PF$-consistent: 
\[
    \Gamma : = \{\BP D, \BF D \mid \BP D \in \ul x\} \cup \{\BP C, - C\}. 
\]
Then, 
\[
    \PF \vdash \bigwedge_{\BP D \in \ul x}\bigl(\BP  D \land \BF D \bigr) \to (\BP C \to C). 
\]
By distributing $\BP$,  
\[
    \PF \vdash \bigwedge_{\BP D \in \ul x}\bigl(\BP \BP  D \land \BP \BF D \bigr) \to \BP (\BP C \to C). 
\]
Since $\PF \vdash \BP D \to \BP \BP D \land \BP \BF D$ and $\PF \vdash \BP (\BP C \to C) \to \BP C$, we obtain
\[
    \PF \vdash \bigwedge \ul x \to \BP C. 
\]
This contradicts $\BP C \notin x$. 

We have shown that $\Gamma$ is $\PF$-consistent, and so $\Gamma$ can be extended to a maximally $\PF$-consistent set $x' \in W_0$. 
Then, $([x'], \ul x') \in W$. 
By the definition of $\Gamma$, it is easy to see $([x], y) \accp ([x'], \ul x')$. 
Since $C \notin \ul x'$, by the induction hypothesis, we obtain $([x'], \ul x') \nVdash C$. 
Therefore, $([x], y) \nVdash \BP C$. 

$(\Leftarrow)$: 
Suppose $\BP C \in y$. 
Since $\neg \BP C \notin y$, by Lemma \ref{useful}, $\BF \neg \BP C \notin x$. 
Then, $\neg \BP C \notin x$ because $\PF \vdash \neg \BP C \to \BF \neg \BP C$. 
Since $x$ is maximally consistent, we obtain $\BP C \in x$. 

Let $([x'], u)$ be any element of $W$ satisfying $([x], y) \accp ([x'], u)$. 
Since $\BP C \in x$, we have $\BF C \in x'$. 
By Lemma \ref{useful}, we get $C \in u$. 
By the induction hypothesis, $([x'], u) \Vdash C$. 
Therefore, we conclude $([x], y) \Vdash \BP C$. 

\paragraph*{Case 2:} $B$ is $\BF C$. 

$(\Rightarrow)$: We prove the contrapositive. 
Assume $\BF C \notin y$. 
Suppose, towards a contradiction, that $\BF \neg \bigwedge \Gamma \in x$ for the set $\Gamma$ defined as follows: 
\[
    \Gamma : = \{\BF D \mid \BF D \in y\} \cup \{- C\}. 
\]
Then, 
\[
    \BF \bigl(\bigwedge_{\BF D \in y} \BF D \to C \bigr) \in x. 
\]
In the similar way as above, we obtain
\[
    \BF \bigl(\bigwedge y \to \BF C \bigr) \in x. 
\]
Thus, we have $\BF \neg \bigwedge y \in x$ because $\neg \BF C \in y$ by the $A$-maximality of $y$. 
This contradicts the $\PF$-consistency of $x$ because $\DF \bigwedge y \in x$. 

We have shown that $\DF \bigwedge \Gamma \in x$. 
Since $\Gamma \subseteq \Subb(A)$, we find some $z \in W_1$ such that $\Gamma \subseteq z$ and $\DF \bigwedge z \in x$. 
Then, $([x], \ul z) \in W$.
By the definition of $\Gamma$, we have $([x], y) \accf ([x], z)$. 
Since $C \notin z$, by the induction hypothesis, $([x], z) \nVdash C$. 
We conclude $([x], y) \nVdash \BF C$. 

$(\Leftarrow)$: Suppose $\BF C \in y$. 
Let $([x], z)$ be any element of $W$ satisfying $([x], y) \accf ([x], z)$. 
Then, $\BF C \in z$, and hence $C \in z$ because $\PF \vdash \BF C \to C$. 
By the induction hypothesis, $([x], z) \Vdash C$. 
Hence, $([x], y) \Vdash \BF C$. 
\end{proof}



Since $\PF \nvdash A$, we have that $\{-A\}$ is $\PF$-consistent. 
We fix some $w \in W_0$ such that $-A \in w$. 
Then, $([w], \ul w) \in W$ and $A \notin \ul w$. 
By Claim \ref{claim_tl}, $([w], \ul w) \nVdash A$. 
We have proved that $\mc M$ is a finite nice $\PF$-model in which $A$ is not valid. 

Next, we transform $\mc M$ to a finite rooted nice PBA $\PF$-model $\mc M^*$. 
For this purpose, we introduce a translation $\dagger$ of formulas. 
For each $\BP C \in \Subb(A)$, we prepare a distinct propositional variable $q_C$ not contained in $\Subb(A)$. 
We define the translation $\dagger$ of elements of $\Subb(A)$ to $\LF$-formulas recursively as follows: 
\begin{enumerate}
    \item $p^\dagger$ is $p$ for each propositional variable $p$, 
    \item $\dagger$ commutes with each propositional connective and $\BF$, 
    \item $(\BP C)^\dagger$ is $q_C$. 
\end{enumerate}
Since $q_C$ is not contained in $\Subb(A)$, we may freely define the truth value of $q_C$ in $\mc M$ for each $\BP C \in \Subb(A)$. 
For each $([x], y) \in W$, let $([x], y) \Vdash q_C$ if and only if $([x], y) \Vdash \BP C$. 
Then, it is easily proved that $B \leftrightarrow B^\dagger$ is valid in $\mc M$ for every $B \in \Subb(A)$.  
For each $[x] \in W_0 / \equiv$, let $\Th_{[x]}^\dagger$ be the finite set
\[
    \{B^\dagger \mid ([x], \ul x) \Vdash B\ \&\ B \in \Subb(A) \} \cup \{q_C \to \BF q_C, \neg q_C \to \BF \neg q_C \mid \BP C \in \Subb(A)\}. 
\]
Since $([x], \ul x) \Vdash \bigwedge \Th_{[x]}^\dagger$, we have $([x], \ul x) \nVdash \neg \bigwedge \Th_{[x]}^\dagger$, and hence $\SF \nvdash \neg \bigwedge \Th_{[x]}^\dagger$. 
By Theorem \ref{PBA}, we find a rooted $\SF$-model $(\mc C_{[x]}^*, \accf_{[x]}^*, \Vdash_{[x]}^*)$ based on a finite pre-Boolean algebra and a root element $r_{[x]}^*$ of $\mc C_{[x]}^*$ such that $r_{[x]}^* \Vdash_{[x]}^* \bigwedge \Th_{[x]}^\dagger$. 
We may assume that $\mc C_{[x]}^* \cap \mc C_{[x']}^* = \emptyset$ for $[x] \neq [x']$. 

Let $\mc M^* = (W^*, \accp^*, \accf^*, \Vdash^*)$ be the Kripke model obtained from the $\mc C_{[w]}$-generated submodel of $\mc M$ by simply replacing each cluster $\mc C_{[x]}$ with the corresponding pre-Boolean algebra $\mc C_{[x]}^*$. 
More precisely, the Kripke model $\mc M^* = (W^*, \accp^*, \accf^*, \Vdash^*)$ is defined as follows: 
\begin{itemize}
    \item $W^* = \bigcup \{\mc C_{[x]}^* \mid [x] = [w]\ \text{or}\ [w] \accp [x]\}$, 
    \item $\accp^* = \bigcup \{\mc C_{[x]}^* \times \mc C_{[x']}^* \mid ([x] = [w]\ \text{or}\ [w] \accp [x])\ \&\ [x] \accp [x']\}$, 
    \item $\accf^* = \bigcup \{ \accf_{[x]}^* \mid [x] = [w]\ \text{or}\ [w] \accp [x]\}$, 
    \item For $a \in \mc C_{[x]}^*$, let $a \Vdash^* p : \iff a \Vdash_{[x]}^* p$. 
\end{itemize}

It is easily shown that our model $\mc M^*$ is a finite rooted nice PBA $\PF$-model with the $\accp^*$-root cluster $\mc C_{[w]}^*$. 
Also, it is proved that for any $a \in \mc C_{[x]}^* \subseteq W^*$ and $\LF$-formula $B$, we have that $a \Vdash^* B$ in $\mc M^*$ if and only if $a \Vdash_{[x]}^* B$ in $(\mc C_{[x]}^*, \accf_{[x]}^*, \Vdash_{[x]}^*)$. 
We define the modal degree $d(B)$ of $\LPF$-formulas $B$ recursively as follows: 
\begin{itemize}
    \item $d(p) : = 0$, 
    \item $d(B \circ C) : = \max \{d(B), d(C)\}$ for $\circ \in \{\land, \lor, \to\}$, 
    \item $d(\neg B) = d(\BF B) : = d(B)$, 
    \item $d(\BP B) : = d(B) + 1$.  
\end{itemize}

\begin{cl}\label{claim_transform}
    For any $B \in \Subb(A)$ and $a \in W^*$, we have $a \Vdash^* B \leftrightarrow B^\dagger$. 
\end{cl}
\begin{proof}
We prove the claim by induction on $d(B)$. 
The case of $d(B) = 0$ is trivial because $B$ is an $\LF$-formula and $B^\dagger$ is $B$ itself. 
We prove the case of $d(B) = n+1$. 
Since $\dagger$ commutes with each propositional connective and $\BF$, it suffices to show the case that $B$ is of the form $\BP C$ where the claim holds for $C$. 
We prove $a \Vdash^* \BP C \leftrightarrow q_C$ for $a \in \mc C_{[x]}^* \subseteq W^*$. 

$(\rightarrow)$: 
Suppose $a \Vdash^* \BP C$. 
Then, $a \Vdash^* \BP \BF C$. 
Let $\mc C_{[x']}^* \subseteq W^*$ be any pre-Boolean algebra such that $[x] \accp^* [x']$. 
Since $a \accp^* r_{[x']}^*$, we have $r_{[x']}^* \Vdash^* \BF C$. 
So, for any $b \in \mc C_{[x']}^*$, we get $b \Vdash^* C$. 
By the induction hypothesis, $b \Vdash^* C^\dagger$. 
Since $r_{[x']}^*$ is a root element of $\mc C_{[x']}^*$, we obtain $r_{[x']}^* \Vdash^* (\BF C)^\dagger$. 
Since $\BF C \in \Subb(A)$, we have $(\BF C)^\dagger \in \Th_{[x']}^\dagger$, and thus $([x'], \ul x') \Vdash \BF C$. 
Therefore, $([x'], y) \Vdash C$ for all $([x'], y) \in \mc C_{[x']}$. 
We have proved that for any $([x'], y) \in W$, if $([x], \ul x) \accp ([x'], y)$, then $([x'], y) \Vdash C$. 
Hence, $([x], \ul x) \Vdash \BP C$. 
Since $q_C$ and $q_C \to \BF q_C$ are in $\Th_{[x]}^\dagger$, we obtain $r_{[x]}^* \Vdash^* \BF q_C$. 
Since $r_{[x]}^* \accf^* a$, we conclude $a \Vdash^* q_C$.

$(\leftarrow)$: 
Suppose $a \nVdash^* \BP C$. 
Then, we find some $b \in \mc C_{[x']}^* \subseteq W^*$ such that $a \accp^* b$ and $b \nVdash^* C$. 
In this case, $[x] \accp^* [x']$. 
By the induction hypothesis, $b \nVdash^* C^\dagger$, and hence $r_{[x']}^* \nVdash^* (\BF C)^\dagger$. 
Since $\neg \BF C \in \Subb(A)$, we have $(\neg \BF C)^\dagger \in \Th_{[x']}^\dagger$, and thus $([x'], \ul x') \nVdash \BF C$. 
Since $([x], \ul x) \accp ([x'], \ul x')$, we obtain $([x], \ul x) \nVdash \BP \BF C$ and hence $([x], \ul x) \nVdash \BP C$. 
Since $\neg q_C$ and $\neg q_C \to \BF \neg q_C$ are in $\Th_{[x]}^\dagger$, we obtain $r_{[x]}^* \Vdash^* \BF \neg q_C$. 
Since $r_{[x]}^* \accf^* a$, we conclude $a \nVdash^* q_C$.
\end{proof}

Since $([w], \ul w) \nVdash A$, we have $\neg A \in \Th_{[w]}^\dagger$, and thus $r_{[w]}^* \nVdash^* A^\dagger$. 
By Claim \ref{claim_transform}, we obtain $r_{[w]}^* \nVdash^* A$. 
Therefore, $\mc M^*$ is a finite rooted nice PBA $\PF$-model in which $A$ is not valid. 
The proof of Theorem \ref{ffp_PF} was completed. 
\end{proof}

\begin{cor}[The decidability of $\PF$]\label{dec_PF}
The logic $\PF$ is decidable. 
\end{cor}

\section{The main theorem}\label{sec_mt}

In this section, we prove that if $\ZFC$ is arithmetically $\Sigma_1$-sound, then the modal logic of provability and forcing is exactly the logic $\PF$. 
This is the main theorem of the present paper. 

\begin{thm}\label{MT}
Suppose that $\ZFC$ is arithmetically $\Sigma_1$-sound. 
Then, for any $\LPF$-formula $A$, the following are equivalent: 
\begin{enumerate}
    \item $\PF \vdash A$. 
    \item $\ZFC \vdash \gpf(A)$ for all translations $g$. 
\end{enumerate}
\end{thm}
\begin{proof}
$(1 \Rightarrow 2)$: By Proposition \ref{soundness_PF}. 

$(2 \Rightarrow 1)$: We prove the contrapositive. 
Suppose $\PF \nvdash A$. 
By Theorem \ref{ffp_PF}, we find a finite rooted nice PBA $\PF$-model $\mc M = (W, \accp, \accf, \Vdash)$ and a $\accf$-root element $r$ of the $\accp$-root cluster $\mc C_r$ such that $r \nVdash A$. 
We extend $\mc M$ to the Kripke model $\mc M^* = (W^*, \accp^*, \accf^*, \Vdash^*)$ by adding one new element $0$ at the bottom as follows: 
\begin{itemize}
    \item $W^* = W \cup \{0\}$, 
    \item $\accp^* = \accp \cup \{(0, x) \mid x \in W\}$, 
    \item $\accf^* = \accf \cup \{(0, 0)\}$, 
    \item for $a \in W$, $a \Vdash^* p$ iff $a \Vdash p$; and $0 \Vdash^* p$ is arbitrary. 
\end{itemize}
Then, $\mc M^*$ is also a finite rooted nice PBA $\PF$-model and $r \nVdash^* A$. 
Let $\{\lambda_i\}$ be Solovay sentences of the finite rooted $\GL$-frame $(W^*/ {\sim^+}, \accp^*)$ (cf.~Theorem \ref{SS}), where $i$ runs over clusters of $\mc M^*$. 
For each cluster $i$, let $\{\chi_a^i\}$ be HL sentences of the finite pre-Boolean algebra $(i, \accf^*)$ (cf.~Theorem \ref{HLS}). 
We define the translation $g$ as follows: 
\[
    g(p) = \bigvee_{\substack{a \in i \\ a \Vdash^* p}} (\lambda_i \land \chi_a^i).
\]

We prove the following claim: 

\begin{cl*}
    For any cluster $i \neq \mc C_0$, $a \in i$, and $\LPF$-formula $B$, 
\begin{enumerate}
    \item If $a \Vdash^* B$, then $\ZFC \vdash \lambda_i \land \chi_a^i \to \gpf(B)$; 
    \item If $a \nVdash^* B$, then $\ZFC \vdash \lambda_i \land \chi_a^i \to \neg \gpf(B)$.  
\end{enumerate}
\end{cl*}
\begin{proof}
We simultaneously prove clauses 1 and 2 by induction on the construction of $B$. 
\begin{itemize}
    \item $B$ is a propositional variable $p$. \\
    (1). Suppose $a \Vdash^* p$, then $\lambda_i \land \chi_a^i$ is a disjunct of $g(p)$. 
    Thus, $\ZFC \vdash \lambda_i \land \chi_a^i \to \gpf(p)$. 

    (2). Suppose $a \nVdash^* p$. 
    Let $\lambda_j \land \chi_b^j$ be any disjunct of $g(p)$. 
    If $i \neq j$, then $\ZFC \vdash \lambda_i \to \neg \lambda_j$ by Theorem \ref{SS}.(1). 
    If $i = j$, then $b \neq a$ because $b \Vdash^* p$, and hence $\ZFC \vdash \chi_a^i \to \neg \chi_b^i$ by Theorem \ref{HLS}.(1). 
    In either case, we have $\ZFC \vdash \lambda_i \land \chi_a^i \to \neg (\lambda_j \land \chi_b^j)$. 
    We conclude $\ZFC \vdash \lambda_i \land \chi_a^i \to \neg \gpf(p)$. 
\end{itemize}
The cases of $\neg$, $\land$, $\lor$, and $\to$ follow from the induction hypothesis. 
It suffices to prove the cases that $B$ is $\BP C$ and $B$ is $\BF C$, where the claim holds for $C$. 

\begin{itemize}
    \item $B$ is $\BP C$. \\
    (1). Suppose $a \Vdash^* \BP C$. 
        Let $j$ be any cluster with $i \accp^* j$ and let $b \in j$ be any element. 
        Then, we have $a \accp^* b$, and so $b \Vdash^* C$. 
        By the induction hypothesis, $\ZFC \vdash \lambda_j \land \chi_b^j \to \gpf(C)$. 
        Because $b \in j$ is arbitrary, $\ZFC \vdash \lambda_j \land \bigvee_{b \in j} \chi_b^j \to \gpf(C)$. 
        Since $\ZFC \vdash \bigvee_{b \in j} \chi_b^j$ by Theorem \ref{HLS}.(2), we obtain $\ZFC \vdash \lambda_j \to \gpf(C)$. 
        We also have $\ZFC \vdash \bigvee_{i \accp^* j} \lambda_j \to \gpf(C)$, and so
        \[
            \ZFC \vdash \Prov \bigvee_{i \accp^* j} \lambda_j \to \gpf(\BP C).
        \]
        Since $\ZFC \vdash \lambda_i \to \Prov \bigvee_{i \accp^* j} \lambda_j$ by Theorem \ref{SS}.(3), we conclude $\ZFC \vdash \lambda_i \to \gpf(\BP C)$. 
        
    (2). Suppose $a \nVdash^* \BP C$, then we find a cluster $j$ and $b \in j$ such that $i \accp^* j$ and $b \nVdash^* C$. 
    By the induction hypothesis, $\ZFC \vdash \lambda_j \land \chi_b^j \to \neg \gpf(C)$. 
    Then, $\ZFC \vdash \Con (\lambda_j \land \chi_b^j) \to \neg \gpf(\BP C)$.
    By Proposition \ref{con_HL}, we have $\ZFC \vdash \Con \lambda_j \to \Con (\lambda_j \land \chi_b^j)$, and hence $\ZFC \vdash \Con \lambda_j \to \neg \gpf(\BP C)$.
    Since $\ZFC \vdash \lambda_i \to \Con \lambda_j$ by Theorem \ref{SS}.(4), we conclude $\ZFC \vdash \lambda_i \to \neg \gpf(\BP C)$. 

    \item $B$ is $\BF C$. \\
    (1). Suppose $a \Vdash^* \BF C$. 
    Let $b \in i$ be such that $a \accf^* b$. 
    Since $b \Vdash^* C$, by the induction hypothesis, $\ZFC \vdash \lambda_i \land \chi_b^i \to \gpf(C)$. 
    Since $b$ is arbitrary, $\ZFC \vdash \lambda_i \land \bigvee_{a \accf^* b} \chi_b^i \to \gpf(C)$. 
    So, we obtain
    \[
        \ZFC \vdash \ForceB \lambda_i \land \ForceB \bigvee_{a \accf^* b} \chi_{b}^i \to \gpf(\BF C). 
    \]
    By Proposition \ref{prop_arith}, $\ZFC \vdash \lambda_i \to \ForceB \lambda_i$ because $\lambda_i$ is arithmetical. 
    Also, we have $\ZFC \vdash \chi_{a}^i \to \ForceB \bigvee_{a \accf^* b} \chi_{b}^i$ by Theorem \ref{HLS}.(3). 
    Thus, we conclude $\ZFC \vdash \lambda_i \land \chi_{a}^i \to \gpf(\BF C)$. 
    
    (2). Suppose $a \nVdash^* \BF C$, then we find a $b \in i$ such that $a \accf^* b$ and $b \nVdash^* C$.
    By the induction hypothesis, $\ZFC \vdash \lambda_i \land \chi_{b}^i \to \neg \gpf(C)$, and then
    \[
        \ZFC \vdash \ForceD (\lambda_i \land \chi_{b}^i) \to \neg \gpf(\BF C). 
    \]
    Since $a \accf^* b$, we have $\ZFC \vdash \chi_{a}^i \to \ForceD \chi_{b}^i$ by Theorem \ref{HLS}.(4). 
    By Proposition \ref{prop_arith}, $\ZFC \vdash \lambda_i \to \ForceB \lambda_i$, and thus $\ZFC \vdash \lambda_i \land \chi_{a}^i \to \ForceD (\lambda_i \land \chi_{b}^i)$. 
    Therefore, we conclude $\ZFC \vdash \lambda_i \land \chi_{a}^i \to \neg \gpf(\BF C)$. \qedhere
    \end{itemize}
\end{proof}
Since $r \nVdash^* A$, by the claim, $\ZFC \vdash \lambda_{\mc C_r} \land \chi_{r}^{\mc C_r} \to \neg \gpf(A)$. 
Thus, $\ZFC \vdash \Con (\lambda_{\mc C_r} \land \chi_{r}^{\mc C_r}) \to \neg \gpf(\BP A)$. 
Since $\mc C_0 \accp^* \mc C_r$, we have $\ZFC \vdash \lambda_{\mc C_0} \to \Con \lambda_{\mc C_r}$ by Theorem \ref{SS}.(3). 
By Theorem \ref{con_HL}, $\ZFC \vdash \Con \lambda_{\mc C_r} \to \Con (\lambda_{\mc C_r} \land \chi_{r}^{\mc C_r})$. 
Therefore, we obtain $\ZFC \vdash \lambda_{\mc C_0} \to \neg \gpf(\BP A)$. 
Since $\ZFC \nvdash \neg \lambda_{\mc C_0}$ by Theorem \ref{SS}.(6), we have $\ZFC \nvdash \gpf(\BP A)$. 
In particular, $\ZFC \nvdash \gpf(A)$. 
\end{proof}

\section{The logic $\PFo$ for $\omega$-models}\label{sec_pfo}

In addition to the study of $\GL$, Solovay also investigated the truth provability logic. 

\begin{defn}
The logic $\mathbf{S}$ in the language $\LP$ is defined as follows: 
    \begin{itemize}
        \item The axioms of $\mathbf{S}$ are:
        \begin{enumerate}
            \item All theorems of $\GL$. 
            \item $\BP A \to A$. 
        \end{enumerate}
        \item The sole inference rule of $\mathbf{S}$ is modus ponens. 
    \end{itemize}
\end{defn}

Actually, Solovay proved that $\mathbf{S}$ is exactly the modal logic of the provability true in the standard model of arithmetic. 
We say a model $M$ of $\ZFC$ is an \textit{$\omega$-model} iff $\omega^M$ is the set of all standard natural numbers. 
In the context of set theory, Solovay's second theorem is stated as follows: 

\begin{thm}[Solovay \cite{Sol76}]
For any $\omega$-model $M$ of $\ZFC$ and $\LP$-formula $A$, the following are equivalent: 
\begin{enumerate}
    \item $\mathbf{S} \vdash A$. 
    \item $M \vDash \gp(A)$ for all translations $g$. 
\end{enumerate}
\end{thm}

Inspired by this observation, in this section, we study the modal logic of provability and forcing true in $\omega$-models. 
We introduce the logic $\PFo$ which is an analogue of $\mathbf{S}$. 

\begin{defn}
    We define the logic $\PFo$ in the language $\LPF$ as follows: 
    \begin{itemize}
        \item The axioms of $\PFo$ are:
        \begin{enumerate}
            \item All theorems of $\PF$. 
            \item $\BP A \to \BF A$. 
        \end{enumerate}
        \item The sole inference rule of $\PFo$ is modus ponens. 
    \end{itemize}
\end{defn}

The following proposition says that the axiom $\BP A \to \BF A$ of $\PFo$ can be replaced by $\BF (\BP A \to A)$. 

\begin{prop}
For any $\LPF$-formula $A$, 
\[
    \PF \vdash (\BP A \to \BF A) \leftrightarrow \BF (\BP A \to A).
\]
\end{prop}
\begin{proof}
$(\rightarrow)$: Since $\PF \vdash (\BP A \to \BF A) \land \BP A \to \BF A$, we have $\PF \vdash (\BP A \to \BF A) \land \BP A \to \BF (\BP A \to A)$. 
On the other hand, $\PF \vdash \neg \BP A \to \BF \neg \BP A$, and hence $\PF \vdash \neg \BP A \to \BF (\BP A \to A)$. 
By the law of excluded middle, we conclude $\PF \vdash (\BP A \to \BF A) \to \BF (\BP A \to A)$. 

$(\leftarrow)$: This implication follows from $\PF \vdash \BF (\BP A \to A) \to (\BF \BP A \to \BF A)$ and $\PF \vdash \BP A \to \BF \BP A$. 
\end{proof}

\begin{prop}\label{soundness_PFo}
Let $M$ be any $\omega$-model of $\ZFC$, $A$ be any $\LPF$-formula, and $g$ be any translation. 
If $\PFo \vdash A$, then $M \vDash \gpf(A)$. 
\end{prop}
\begin{proof}
We prove the proposition by induction on the length of proofs of $A$ in $\PFo$. 
\begin{itemize}
    \item If $A$ is a theorem of $\PF$, then by Proposition \ref{soundness_PF}, we have $\ZFC \vdash \gpf(A)$. 
    Since $M$ is a model of $\ZFC$, we get $M \vDash \gpf(A)$. 

    \item Suppose that $A$ is of the form $\BP B \to \BF B$. 
    Assume $M \vDash \gpf(\BP B)$. 
    Then, $M \vDash \Prov \gpf(B)$. 
    Since $M$ is an $\omega$-model, this implies that $\ZFC \vdash \gpf(B)$. 
    Then, $\ZFC \vdash \gpf(\BF B)$. 
    Since $M \vDash \ZFC$, we obtain $M \vDash \gpf(\BF B)$. 

    \item Suppose that $\PFo \vdash B$, $\PFo \vdash B \to A$, and the proposition holds for $B$ and $B \to A$. 
    Then, $M \vDash \gp(B)$ and $M \vDash \gpf(B \to A)$. 
    We have $M \vDash \gpf(A)$. \qedhere
\end{itemize}
\end{proof}

\begin{defn}
    For each formula $A$, let $\Phi(A)$ be the set
    \[
        \{\BP B \to \BF B \mid \BP B\ \in \Sub(A)\}. 
    \]
\end{defn}

We prove that the modal logic of provability and forcing of every $\omega$-model of $\ZFC + \mathrm{V} = \CU$ is exactly the logic $\PFo$. 

\begin{thm}\label{compl_PFo1}
    Let $M$ be an $\omega$-model of $\ZFC + \mathrm{V} = \CU$. 
    For any $\LPF$-formula $A$, the following are equivalent: 
\begin{enumerate}
    \item $\PF \vdash \bigwedge \Phi(A) \to A$. 
    \item $\PF^\omega \vdash A$.
    \item $M \vDash \gpf(A)$ for any translation $g$. 
\end{enumerate}
\end{thm}
\begin{proof}
$(1 \Rightarrow 2)$: Suppose $\PF \vdash \bigwedge \Phi(A) \to A$. 
Then, $\PFo \vdash \bigwedge \Phi(A) \to A$. 
Since $\PFo \vdash \bigwedge \Phi(A)$, we conclude $\PFo \vdash A$. 

$(2 \Rightarrow 3)$: By Proposition \ref{soundness_PFo}. 

$(3 \Rightarrow 1)$: We prove the contrapositive. 
Suppose $\PFo \nvdash \bigwedge \Phi(A) \to A$. 
By Theorem \ref{ffp_PF}, we find a finite rooted nice PBA $\PF$-model $\mc M = (W, \accp, \accf, \Vdash)$ and a $\accf$-root element $r$ of the $\accp$-root cluster $\mc C_r$ such that $r \Vdash \bigwedge \Phi(A) \land \neg A$. 
For each $a \in \mc C_r$, we prepare a copy $a^*$ of $a$, and let $\mc C^* : = \{a^* \mid a \in \mc C_r\}$. 
We extend $\mc M$ to the Kripke model $\mc M^* = (W^*, \accp^*, \accf^*, \Vdash^*)$ as follows: 
\begin{itemize}
    \item $W^* = W \cup \mc C^*$, 
    \item $\accp^* = \accp \cup \{(a^*, x) \mid a \in \mc C_r$ and $x \in W\}$, 
    \item $\accf^* = \accf \cup \{(a^*, b^*) \mid a \accf b\}$, 
    \item for $a \in W$, $a \Vdash^* p$ iff $a \Vdash p$; and for $a \in \mc C_r$, $a^* \Vdash^* p$ iff $a \Vdash p$. 
\end{itemize}
It is shown that $\mc M^*$ is also a finite rooted nice PBA $\PF$-model and $r \Vdash^* \bigwedge \Phi(A) \land \neg A$. 
Let $\{\lambda_i\}$ be Solovay sentences of the finite rooted $\GL$-frame $(W^*/ {\sim^+}, \accp^*)$, where $i$ runs over clusters of $\mc M^*$. 
For each cluster $i$, let $\{\chi_a^i\}$ be HL sentences of the finite pre-Boolean algebra $(i, \accf^*)$. 
Let $g$ be the translation defined by 
\[
    g(p) = \bigvee_{\substack{a \in i \\ a \Vdash^* p}} (\lambda_i \land \chi_a^i).
\]

By the claim in the proof of Theorem \ref{MT}, we have already proved that for any $i \neq \mc C^*$, $a \in i$, and formula $B$, 
\begin{enumerate}
    \item If $a \Vdash^* B$, then $\ZFC \vdash \lambda_i \land \chi_a^i \to \gpf(B)$; 
    \item If $a \nVdash^* B$, then $\ZFC \vdash \lambda_i \land \chi_a^i \to \neg \gpf(B)$.  
\end{enumerate}

We prove the following claim: 

\begin{cl*}
    For any $a \in \mc C_r$ and any subformula $B$ of $A$, 
\begin{enumerate}
    \item If $a \Vdash^* B$, then $\ZFC \vdash \lambda_{\mc C^*} \land \chi_{a^*}^{\mc C^*} \to \gpf(B)$; 
    \item If $a \nVdash^* B$, then $\ZFC \vdash \lambda_{\mc C^*} \land \chi_{a^*}^{\mc C^*} \to \neg \gpf(B)$.  
\end{enumerate}
\end{cl*}
\begin{proof}
We simultaneously prove clauses 1 and 2 by induction on the construction of $B$. 
The case that $B$ is a propositional variable $p$ is proved in the similar way as in the proof of Theorem \ref{MT} by paying attention to the equivalence between $a^* \Vdash^* p$ and $a \Vdash^* p$. 
We prove only the cases that $B$ is $\BP C$ and $B$ is $\BF C$, where the claim holds for $C$. 

\begin{itemize}
    \item $B$ is $\BP C$. \\
    (1). Suppose $a \Vdash^* \BP C$. 
        Let $i$ be any cluster with $\mc C_r \accp^* i$ and let $b \in i$ be any element. 
        Then, we have $a \accp^* b$, and so $b \Vdash^* C$. 
        Since $i \neq \mc C^*$, we obtain $\ZFC \vdash \lambda_i \land \chi_b^i \to \gpf(C)$. 
        Then, $\ZFC \vdash \lambda_i \land \bigvee_{b \in i} \chi_b^i \to \gpf(C)$ because $b \in i$ is arbitrary. 
        Since $\ZFC$ proves $\bigvee_{b \in i} \chi_b^i$ by Theorem \ref{HLS}.(2), we obtain $\ZFC \vdash \lambda_i \to \gpf(C)$. 
        Hence, 
        \begin{equation}\label{BP}
            \ZFC \vdash \bigvee_{\mc C_r \accp^* i} \lambda_i \to \gpf(C).
        \end{equation}

        Since $\BP C$ is a subformula of $A$, it follows $r \Vdash^* \BP C \to \BF C$ from $r \Vdash^* \bigwedge \Phi(A)$. 
        For any $x \in W$ with $r \accp^* x$, we have $a \accp^* x$, so $x \Vdash^* C$. 
        So $r \Vdash^* \BP C$, and hence $r \Vdash^* \BF C$. 
        Let $b$ be an arbitrary element of $\mc C_r$. 
        Since $r \accf^* b$, we have $b \Vdash^* C$. 
        We get $\ZFC \vdash \lambda_{\mc C_r} \land \chi_b^{\mc C_r} \to \gpf(C)$. 
        Also, by the induction hypothesis, $\ZFC \vdash \lambda_{\mc C^*} \land \chi_{b^*}^{\mc C^*} \to \gpf(C)$. 
        Since $b \in \mc C_r$ is arbitrary, we obtain
        \[
            \ZFC \vdash \lambda_{\mc C_r} \land \bigvee_{b \in \mc C_r} \chi_b^{\mc C_r} \to \gpf(C)
        \]
        and 
        \[
            \ZFC \vdash \lambda_{\mc C^*} \land \bigvee_{b^* \in \mc C^*} \chi_{b^*}^{\mc C^*} \to \gpf(C). 
        \]
        Since $\ZFC$ proves $\bigvee_{b \in \mc C_r} \chi_b^{\mc C_r}$ and $\bigvee_{b^* \in \mc C^*} \chi_{b^*}^{\mc C^*}$, we have $\ZFC \vdash \lambda_{\mc C_r} \to \gpf(C)$ and $\ZFC \vdash \lambda_{\mc C^*} \to \gpf(C)$. 
        By combining them with (\ref{BP}), we obtain $\ZFC \vdash \bigvee_{i} \lambda_i \to \gpf(C)$. 
        Since $\ZFC$ also proves $\bigvee_{i} \lambda_i$ by Theorem \ref{SS}.(2), we get $\ZFC \vdash \gpf(C)$. 
        We conclude $\ZFC \vdash \gpf(\BP C)$. 
        
    (2). Suppose $a \nVdash^* \BP C$, then we find a cluster $i$ and an element $b \in i$ such that $\mc C_r \accp^* i$ and $b \nVdash^* C$. 
    Since $i \neq \mc C^*$, we have $\ZFC \vdash \lambda_i \land \chi_b^i \to \neg \gpf(C)$. 
    Then, $\ZFC \vdash \Con (\lambda_i \land \chi_b^i) \to \neg \gpf(\BP C)$.
    By Theorem \ref{con_HL}, we have $\ZFC \vdash \Con \lambda_i \to \Con (\lambda_i \land \chi_b^i)$, and hence $\ZFC \vdash \Con \lambda_i \to \neg \gpf(\BP C)$.
    Since $\mc C^* \accp^* i$, we get $\ZFC \vdash \lambda_{\mc C^*} \to \Con \lambda_i$ by Theorem \ref{SS}.(4). 
    Thus, we conclude $\ZFC \vdash \lambda_{\mc C^*} \to \neg \gpf(\BP C)$. 

    \item $B$ is $\BF C$. \\
    (1). Suppose $a \Vdash^* \BF C$. 
    Let $b \in \mc C_r$ be any element with $a \accf^* b$, then $b \Vdash^* C$. 
    By the induction hypothesis, we have $\ZFC \vdash \lambda_{\mc C^*} \land \chi_{b^*}^{\mc C^*} \to \gpf(C)$. 
    Since $b$ is arbitrary, $\ZFC \vdash \lambda_{\mc C^*} \land \bigvee_{a \accf^* b} \chi_{b^*}^{\mc C^*} \to \gpf(C)$. 
    Since $a \accf^* b$ is equivalent to $a^* \accf^* b^*$, we have $\ZFC \vdash \lambda_{\mc C^*} \land \bigvee_{a^* \accf^* b^*} \chi_{b^*}^{\mc C^*} \to \gpf(C)$. 
    So, We obtain
    \[
        \ZFC \vdash \ForceB \lambda_{\mc C^*} \land \ForceB \bigvee_{a^* \accf^* b^*} \chi_{b^*}^{\mc C^*} \to \gpf(\BF C). 
    \]
    Since $\ZFC \vdash \lambda_{\mc C^*} \to \ForceB \lambda_{\mc C^*}$ by Proposition \ref{prop_arith} and we have $\ZFC \vdash \chi_{a^*}^{\mc C^*} \to \ForceB \bigvee_{a^* \accf^* b^*} \chi_{b^*}^{\mc C^*}$ by Theorem \ref{HLS}.(3), we conclude
    \[
        \ZFC \vdash \lambda_{\mc C^*} \land \chi_{a^*}^{\mc C^*} \to \gpf(\BF C). 
    \]
    
    (2). Suppose $a \nVdash^* \BF C$, then we find a $b \in \mc C_r$ such that $a \accf^* b$ and $b \nVdash^* C$.
    By the induction hypothesis, we have $\ZFC \vdash \lambda_{\mc C^*} \land \chi_{b^*}^{\mc C^*} \to \neg \gpf(C)$, and then \[
        \ZFC \vdash \ForceD (\lambda_{\mc C^*} \land \chi_{b^*}^{\mc C^*}) \to \neg \gpf(\BF C). 
    \]
    Here, since $a^* \accf^* b^*$, we have $\ZFC \vdash \chi_{a^*}^{\mc C^*} \to \ForceD \chi_{b^*}^{\mc C^*}$ by Theorem \ref{HLS}.(4). 
    By Proposition \ref{prop_arith}, $\ZFC \vdash \lambda_{\mc C^*} \to \ForceB \lambda_{\mc C^*}$, and thus $\ZFC \vdash \lambda_{\mc C^*} \land \chi_{a^*}^{\mc C^*} \to \ForceD (\lambda_{\mc C^*} \land \chi_{b^*}^{\mc C^*})$. 
    Therefore, we conclude
    \[
        \ZFC \vdash \lambda_{\mc C^*} \land \chi_{a^*}^{\mc C^*} \to \neg \gpf(\BF C). \qedhere
    \]
    \end{itemize}
\end{proof}
Since $r \nVdash^* A$, by the claim, $\ZFC \vdash \lambda_{\mc C^*} \land \chi_{r^*}^{\mc C^*} \to \neg \gpf(A)$. 
Then, $M \vDash \lambda_{\mc C^*} \land \chi_{r^*}^{\mc C^*} \to \neg \gpf(A)$. 
For each $i \neq \mc C^*$, we have $M \vDash \lambda_i \to \Prov \neg \lambda_i$ by Theorem \ref{SS}.(5). 
Since $M$ is an $\omega$-model, $M \vDash \Prov \neg \lambda_i \to \neg \lambda_i$. 
Hence, $M \vDash \neg \lambda_i$. 
Since $M \vDash \bigwedge_{i \neq \mc C^*} \neg \lambda_i \to \lambda_{\mc C^*}$ by Theorem \ref{SS}.(2), we have $M \vDash \lambda_{\mc C^*}$. 
Also, $M \vDash \chi_{r^*}^{\mc C^*}$ by Theorem \ref{HLS}.(5) because $M \vDash \mathrm{V} = \CU$. 
Therefore, we conclude $M \nvDash \gpf(A)$. 
\end{proof}

The decidability of $\PFo$ immediately follows from the decidability of $\PF$ (Corollary \ref{dec_PF}) and Theorem \ref{compl_PFo1}. 

\begin{cor}
    The logic $\PFo$ is decidable. 
\end{cor}



If an $\omega$-model of $\ZFC$ exists, then $\PFo$ is also the modal logic of provability and forcing for all $\omega$-models of $\ZFC$. 

\begin{cor}\label{compl_PFo2}
Suppose that there exists an $\omega$-model of $\ZFC$. 
    For any $\LPF$-formula $A$, the following are equivalent: 
\begin{enumerate}
    \item $\PF^\omega \vdash A$.
    \item $M \vDash \gpf(A)$ for any $\omega$-model $M$ of $\ZFC$ and any translation $g$. 
\end{enumerate}
\end{cor}
\begin{proof}
$(1 \Rightarrow 2)$: This is exactly Proposition \ref{soundness_PFo}. 

$(2 \Rightarrow 1)$: Suppose that condition 2 holds. 
Let $M$ be an $\omega$-model of $\ZFC$. 
Then, $\mathrm{L}^M$ is an $\omega$-model of $\ZFC + \mathrm{V} = \CU$. 
By the supposition, we have $\CU^M \vDash \gpf(A)$ for all translations $g$. 
By Theorem \ref{compl_PFo1}, we conclude $\PFo \vdash A$. 
\end{proof}

\section*{Acknowledgements}

The second author would like to thank Hiroshi Sakai for his valuable comments and significant suggestions under his supervision. 
The first author was supported by JSPS KAKENHI Grant Numbers JP19K14586 and JP23K03200.

\bibliographystyle{plain}
\bibliography{reference}

\appendix

\section{Some conservation results}\label{App1}

This appendix aims to prove the following conservation results:

\begin{enumerate}
    \item $\PF$ is a conservative extension of $\GL$. 
    \item $\PFo$ is a conservative extension of both $\mathbf{S}$ and $\SF$. 
\end{enumerate}

\begin{defn}
    We define the logic $\GT$ in the language $\LPF$ as follows: 
    \begin{itemize}
        \item The axioms of $\GT$ are:
        \begin{enumerate}
	\item All propositional tautologies in the language $\LPF$.  
	\item All axioms of $\GL$. 
    \item $A \leftrightarrow \BF A$. 
        \end{enumerate}
        \item The inference rules of $\GT$ are modus ponens and necessitation for $\BP$. 
    \end{itemize}
The logic $\GT$ is called the \textit{fusion} of $\GL$ and $\mathbf{Triv}$. 
\end{defn}

It is easy to see that $\PF$ is a sublogic of $\GT$. 

\begin{prop}\label{conservation_GT}
    $\GT$ is a conservative extension of $\GL$. 
\end{prop}
\begin{proof}
    Let $A$ be any $\LP$-formula and suppose $\GL \nvdash A$. 
By Theorem \ref{ffp_GL}, we have a $\GL$-model $(W, \accp, \Vdash)$ in which $A$ is not valid. 
It is easy to see that all theorems of $\GT$ are valid on $(W, \accp, =)$, and hence by extending $\Vdash$ to the language $\LPF$, we obtain a $\GT$-model $(W, \accp, =, \Vdash)$ in which $A$ is not valid. 
Hence, $\GT \nvdash A$. 
\end{proof}

\begin{cor}
$\PF$ is a conservative extension of $\GL$. 
\end{cor}

As a corollary to the proposition, we also obtain the conservation result between $\PFo$ and $\mathbf{S}$. 

\begin{cor}
$\PFo$ is a conservative extension of $\mathbf{S}$. 
\end{cor}
\begin{proof}
Let $A$ be any $\LP$-formula and suppose $\PFo \vdash A$. 
By Theorem \ref{compl_PFo1}, we have $\PF \vdash \bigwedge \Phi(A) \to A$. 
Then, $\GT \vdash \bigwedge \Phi(A) \to A$. 
Let $\Psi(A)$ be the set $\{\BP B \to B \mid \BP B \in \Sub(A)\}$ of $\LP$-formulas, then $\GT \vdash \bigwedge \Psi(A) \to A$. 
By Proposition \ref{conservation_GT}, we obtain $\GL \vdash \bigwedge \Psi(A) \to A$. 
Since $\mathbf{S}$ is an extension of $\GL$ proving $\bigwedge \Psi(A)$, we conclude $\mathbf{S} \vdash A$. 
\end{proof}

Finally, we prove the following conservation result. 






\begin{prop}
$\PFo$ is a conservative extension of $\SF$. 
\end{prop}
\begin{proof}
Let $A$ be any $\LF$-formula and suppose that $\SF \nvdash A$.
By Theorem \ref{ffp_SF}, we find a rooted $\SF$-model $(W, \accf, \Vdash)$ in which $A$ is not true in a root element $r \in W$. 
For each element $a \in W$, we prepare a new element $a^*$ as a copy of $a$. 
We define the Kripke model $\mc M^* = (W^*, \accp^*, \accf^*, \Vdash^*)$ as follows: 
\begin{itemize}
    \item $W^* = W \cup \{a^* \mid a \in W\}$, 
    \item $\accp^* = \{(a, b^*) \mid a, b \in W\}$, 
    \item $\accf^* = \accf \cup \{(a^*, b^*) \mid a \accf b\}$, 
    \item for $a \in W$, $a \Vdash^* p \iff a^* \Vdash^* p \iff a \Vdash p$. 
\end{itemize}
It is easily seen that $\mc M^*$ is a nice $\PF$-model. 
Then, it is shown by induction on the construction of $B$ that for any $\LF$-formula $B$ and $a \in W$, 
\[
    a \Vdash^* B \iff a^* \Vdash^* B \iff a \Vdash B. 
\]
Suppose $r \Vdash^* \BP B$ for an $\LF$-formula $B$. 
Then, for each $a \in W$, we have $r \accp^* a^*$. 
So $a^* \Vdash^* B$, and hence $a \Vdash^* B$. 
Therefore, $r \Vdash^* \BF B$. 
We have shown that $r \Vdash^* \BP B \to \BF B$. 
In particular, $r \Vdash \bigwedge \Phi(A) \land \neg A$. 
So, $\PF \nvdash \bigwedge \Phi(A) \to A$. 
By Theorem \ref{compl_PFo1}, we conclude $\PFo \nvdash A$. 
\end{proof}

Hence, $\PF$ is also a conservative extension of $\SF$. 

\section{Conjunctive normal form theorem}\label{App2}

In this appendix, we show that every $\LPF$-formula is equivalent to an $\LPF$-formula in a version of conjunctive normal form over $\PF$. 

\begin{defn}
We say that an $\LPF$-formula $A$ is in \textit{$\BP$-conjunctive normal form} ($\BP$-CNF) iff $A$ is a conjunction of $\LPF$-formulas of the form
\[
    \BP D_0 \lor \cdots \lor \BP D_{k-1} \lor \DP E \lor F, 
\]
where $F$ is an $\LF$-formula. 
\end{defn}

Before proving our conjunctive normal form theorem, we prepare the following proposition. 
Since it is easily proved, proof is left to the reader. 

\begin{prop}\label{prop_pcnf}
Let $A$ and $B$ be any $\LPF$-formulas. 
    \begin{enumerate}
        \item $\PF \vdash \BF (\BP A \lor B) \leftrightarrow (\BP A \lor \BF B)$, 
        \item $\PF \vdash \BF (\DP A \lor B) \leftrightarrow (\DP A \lor \BF B)$.  
    \end{enumerate}
\end{prop}

Recall that $d(A)$ denotes the maximum number of nesting of $\BP$ in $A$ (see the proof of Theorem \ref{ffp_PF} for the definition.) 

\begin{thm}[$\BP$-CNF theorem]
For any $\LPF$-formula $A$, we can effectively find an $\LPF$-formula $A'$ such that $A'$ is in $\BP$-CNF and $d(A) = d(A')$. 
\end{thm}
\begin{proof}
This theorem is proved by induction on the construction of $A$. 
If $A$ is a propositional variable or of the form $\BP B$, then the theorem trivially holds. 
The cases for propositional connectives are proved by easy calculations of propositional logic. 
So, we only give a proof of the case that $A$ is of the form $\BF B$ and the theorem holds for $B$. 
By the induction hypothesis, we find an $\LPF$-formula $B'$ such that $\PF \vdash B \leftrightarrow B'$, $d(B) = d(B')$ and $B'$ is of the form $C_0 \land \cdots \land C_l$. 
Here, each $C_i$ is of the form
\[
    \BP D_0^i \lor \cdots \lor \BP D_{k_i-1}^i \lor \DP E^i \lor F^i, 
\]
where $F^i$ is an $\LF$-formula. 
In this case, we have $\PF \vdash A \leftrightarrow \BF C_0 \land \cdots \land \BF C_l$. 
For each $i$, let $C_i'$ be the $\LPF$-formula
\[
    \BP D_0^i \lor \cdots \lor \BP D_{k_i-1}^i \lor \DP E^i \lor \BF F^i. 
\]
By applying Proposition \ref{prop_pcnf} repeatedly, we have $\PF \vdash \BF C_i \leftrightarrow C_i'$. 
Let $A'$ be the formula $C_0' \land \cdots \land C_l'$. 
Then, we obtain that  $A'$ is in $\BP$-CNF, $\PF \vdash A \leftrightarrow A'$, and $d(A) = d(A')$. 
\end{proof}

\end{document}